\documentclass[12pt]{amsart}

\usepackage{amsmath}
\usepackage{amsthm}
\usepackage{amssymb}
\usepackage{graphicx, color}
\usepackage{import}
\usepackage{subfigure}
\usepackage[hidelinks,pagebackref,pdftex]{hyperref}

\usepackage[margin=1.2in]{geometry}

\AtBeginDocument{%
   \def\MR#1{}
}

\theoremstyle{plain}
\newtheorem{theorem}{Theorem}[section]
\newtheorem*{theorem*}{Theorem}
\newtheorem{proposition}[theorem]{Proposition}
\newtheorem{lemma}[theorem]{Lemma}
\newtheorem{corollary}[theorem]{Corollary}

\newtheorem*{namedtheorem}{\theoremname}
\newcommand{\theoremname}{testing}

\theoremstyle{definition}
\newtheorem{definition}[theorem]{Definition}
\newtheorem{question}[theorem]{Question}

\newtheorem*{fact*}{Fact}

\newcommand{\QQ}{\mathbb{Q}}

\newcommand{\bdy}{\partial}

\newcommand{\vol}{{\rm{vol}}}
\newcommand{\tw}{{\rm{tw}}}
\newcommand{\calL}{\mathcal{L}}
\newcommand{\calD}{\mathcal{D}}

\begin{document}

\title{Independence of volume and genus $g$ bridge numbers} 

\author{Jessica S.\ Purcell}
\author{Alexander Zupan}






\begin{abstract}
A theorem of Jorgensen and Thurston implies that the volume of a hyperbolic 3--manifold is bounded below by a linear function of its Heegaard genus.  Heegaard surfaces and bridge surfaces often exhibit similar topological behavior; thus it is natural to extend this comparison to ask whether a $(g,b)$-bridge surface for a knot $K$ in $S^3$ carries any geometric information related to the knot exterior.  In this paper, we show that --- unlike in the case of Heegaard splittings --- hyperbolic volume and genus $g$ bridge numbers are completely independent.  That is, for any $g$, we construct explicit sequences of knots with bounded volume and unbounded genus $g$ bridge number, and explicit sequences of knots with bounded genus $g$ bridge number and unbounded volume. 
\end{abstract}

\maketitle

\section{introduction}\label{introduction}

A major theme in 3-manifold research is to connect the geometric and topological invariants of a hyperbolic 3-manifold $Y$.  One archetypal example of such a connection is a celebrated theorem of Jorgensen and Thurston, which implies that the hyperbolic volume of $Y$ is linearly related to a topological invariant involving triangulations of $Y$ (see~\cite{KR}, for example).  Another prominent topological invariant of a closed 3-manifold $Y$ is its Heegaard genus $g(Y)$, the smallest $g$ such that $Y$ admits a genus $g$ surface cutting it into two handlebodies, called a \emph{Heegaard surface}.  From the result of Jorgensen and Thurston, it follows that there is a constant $C$ such that for all hyperbolic manifolds $Y$,
\begin{equation}\label{thurston}
C \cdot g(Y) \leq \vol(Y).
\end{equation}
On the other hand, Heegaard genus is not linearly related to volume; for any fixed genus $g \geq 2$, there is a 3-manifold with Heegaard genus $g$ and arbitrarily large volume~\cite{NS}.

A bridge surface for a knot $K$ in a 3-manifold $Y$ may be viewed as a relative Heegaard surface:  A $(g,b)$-bridge surface $\Sigma$ is a genus $g$ Heegaard surface for $Y$ that cuts $K$ into two collections of $b$ unknotted arcs.  When $Y = S^3$, this definition gives rise to a knot invariant for each $g$, the genus $g$ bridge number $b_g(K)$, i.e.\ the smallest $b$ such that $(S^3,K)$ admits a $(g,b)$-bridge surface.  Genus $g$ bridge numbers are related for various values of $g$, and for this reason they may be collated into a sequence called the \emph{bridge spectrum} $\mathbf{b}(K) = \{b_0(K),b_1(K),\dots \}$~\cite{ZBridge}.  In this context, the classical bridge number of $K$ is $b_0(K)$.

Significant topological evidence supports the claim that Heegaard surfaces and bridge surfaces exhibit similar behavior, and most of the technology developed to better understand Heegaard surfaces can be usefully adapted to the context of bridge surfaces. For example, a notion of high distance for a Heegaard splitting will imply that a manifold is hyperbolic~\cite{Hempel}, and similarly high distance for a bridge splitting also implies hyperbolicity~\cite{BachmanSchleimer}. Existence of high distance Heegaard surfaces and bridge surfaces put restrictions on additional Heegaard and bridge surfaces, respectively~\cite{ScharlemannTomova}, \cite{Tomova}.

In terms of volumes, Zupan demonstrated a connection between the topology of bridge spheres and the volume of 2-bridge knots~\cite{ZPants}. And as noted in Doll~\cite{doll}, in addition to defining genus $g$ bridge numbers $b_g(K)$, which involve fixing $g$ and minimizing $b$, there is a notion of the \emph{$b$-bridge genus} $g_b(K)$ of $K$, which (for fixed $b$) is the smallest $g$ such that a knot $K$ admits a $(g,b)$-splitting.  For any $b$ and any knot $K$, the $b$-bridge genus is bounded above by the Heegaard genus of the exterior $E(K) = S^3-K$ of $K$; thus, the inequality (\ref{thurston}) implies that
\begin{equation}
C \cdot g_b(K) \leq \vol(K).
\end{equation}

Given the information contained in the inequalities (1) and (2) combined with the topological similarities between Heegaard surfaces and bridge surfaces, it is natural to ask the following question:

\begin{question}\label{mainQ}
For the collection of hyperbolic knots $K$ in $S^3$, what is the relationship (if any) between genus $g$ bridge numbers and hyperbolic volumes?
\end{question}

One might expect to find that the genus $g$ bridge number of a knot $K \subset S^3$ yields a lower bound for $\vol(K)$, but this is not the case. In this paper, we show the following. 

\begin{theorem}\label{thm:BoundedVolUnboundedBridge}
For any $g$, there exists a sequence of knots $\{ K_n \}$ and a constant $V$ with $b_g(K_n) \to \infty$ as $n\to \infty$ but $\vol(K_n) < V$ for all $n$.
\end{theorem}

Note that the case $g=0$ follows from results on twisted torus knots in~\cite{BowmanTaylorZupan} and~\cite{CKFNP}.

We give a concrete proof of this result, in which we build a sequence of knots and show they satisfy the requirements of the theorem.  The knots are obtained by performing higher and higher annular Dehn fillings on a link $\calL$. The volume is bounded above by the volume of the parent manifold $S^3-\calL$. The genus $g$ bridge number can be bounded below using a theorem of Baker, Gordon, and Luecke \cite[Theorem~1.2]{BakerGordonLuecke2013}.

In contrast to Theorem~\ref{thm:BoundedVolUnboundedBridge}, we also show genus $g$ bridge numbers do not bound volume from above.  This result is not surprising, but we include it here for completeness.

\begin{proposition}\label{thm:UnboundedVolBoundedBridge}
For any $g,b\geq 1$, there exists a sequence of knots $\{K_n\}$ such that $b_g(K_n) =b$ but $\vol(K_n)\to\infty$ as $n\to\infty$.
\end{proposition}

It follows from the main theorem and proposition that the answer to Question~\ref{mainQ} is that in general, there is no relationship between the hyperbolic volume of a knot complement and its genus $g$ bridge numbers; the two types of invariants are independent measures of complexity.

\subsection{Acknowledgements}
We thank Jesse Johnson and Yoav Moriah for helpful discussions.
Purcell is partially supported by NSF grants DMS--1252687 and DMS-1128155.  Zupan is partially supported by NSF grant DMS--1203988.

\section{Preliminaries}\label{sec:Prelim}

In this section, we set conventions and present definitions that will appear in the remainder of the paper.  We assume that all manifolds are compact and orientable.  Given a knot or link $L$ in $S^3$, let $N(L)$ denote a closed regular neighborhood of $L$, and let $E(L)$, the \emph{exterior} of the link, be defined by $E(L) = \overline{S^3 - N(L)}$.  For any genus $g \geq 0$, a $(g,b)$-bridge splitting of $(S^3,L)$ is a decomposition
\[ (S^3,L) = (V,\alpha) \cup_{\Sigma} (W,\beta),\]
where $V$ and $W$ are genus $g$ handlebodies, $\Sigma = \partial V = \partial W$, and $\alpha$ and $\beta$ are collections of $b$ unknotted arcs in $V$ and $W$, respectively.  The \emph{genus $g$ bridge number} $b_g(K)$ is the smallest $b$ such that $K$ admits a $(g,b)$-bridge splitting.

For a link $L \subset S^3$ and a component $L'$ of $L$, we may perform \emph{Dehn filling} on $L'$ by gluing the boundary of a solid torus $V$ to the boundary of $E(L')$.  Curves in the boundary of $\partial(N(L'))$ are naturally parameterized by the extended rational numbers $\QQ\cup\{\infty\}$, and the number corresponding to the image of a meridian of the solid torus $V$ is called the \emph{slope} of the filling.  In particular, if $L'$ is an unknot which bounds a disk $D$ in $S^3$ and we perform Dehn filling of slope $1/n$ on $E(L')$, the resulting 3-manifold is $S^3$, and the filling has the effect of adding $n$ full twists (or $2n$ half twists) to strands of $L - L'$ that pierce the disk $D$.

Finally, we define a \emph{tunnel system} for a knot $K$ in $S^3$ to be a collection of properly embedded arcs $\Gamma$ in $E(K)$ such that $E(K\cup \Gamma)$ is a handlebody.  The \emph{tunnel number} $t(K)$ is the minimal number of arcs in a tunnel system for $K$.  Note that $t(K) = g(E(K))-1$, where $g(E(K))$ is the Heegaard genus of $K$.

\section{Link descriptions}\label{sec:LinkDesc}

In this section, we find, for every $g$, a concrete, explicit sequence of links with bounded volume but $b_g$ approaching infinity.

Our construction starts with a highly twisted plat projection of a knot as in Figure~\ref{fig:HighlyTwistedPlat}, left, which was defined by Johnson and Moriah~\cite{JohnsonMoriah}. Such a projection may be constructed as follows: Start with a braid with $2m$ strands. Recall that we may write the generators of the braid group on $2m$ strands as $\sigma_1,\dots, \sigma_{2m-1}$, where $\sigma_j$ gives a positive crossing between the $j$-th and $(j+1)$-st strands. We take our braid to be given by the product of $n$ elements $b_1, b_2, \dots, b_n$ in the braid group, where $n$ is odd, of the following form:
\[
b_k = \begin{cases}
  \sigma_2^{a_{k,2}}\sigma_4^{a_{k,4}} \dots \sigma_{2m-2}^{a_{k,2m-2}} & \mbox{if $k$ is odd} \\
  \sigma_1^{a_{k,1}}\sigma_3^{a_{k,3}} \dots \sigma_{2m-1}^{a_{k, 2m-1}} & \mbox{if $k$ is even}
  \end{cases}
\]
Connect the $2m$ strands at the top of the braid by $m$ simple arcs, and similarly at the bottom of the braid. The result is a link diagram with a grid of twist regions. The twist region in the $(i,j)$-th position has $a_{i,j}$ crossings. In order to obtain the desired properties, we make the following further requirements on the diagram.

\begin{figure}
  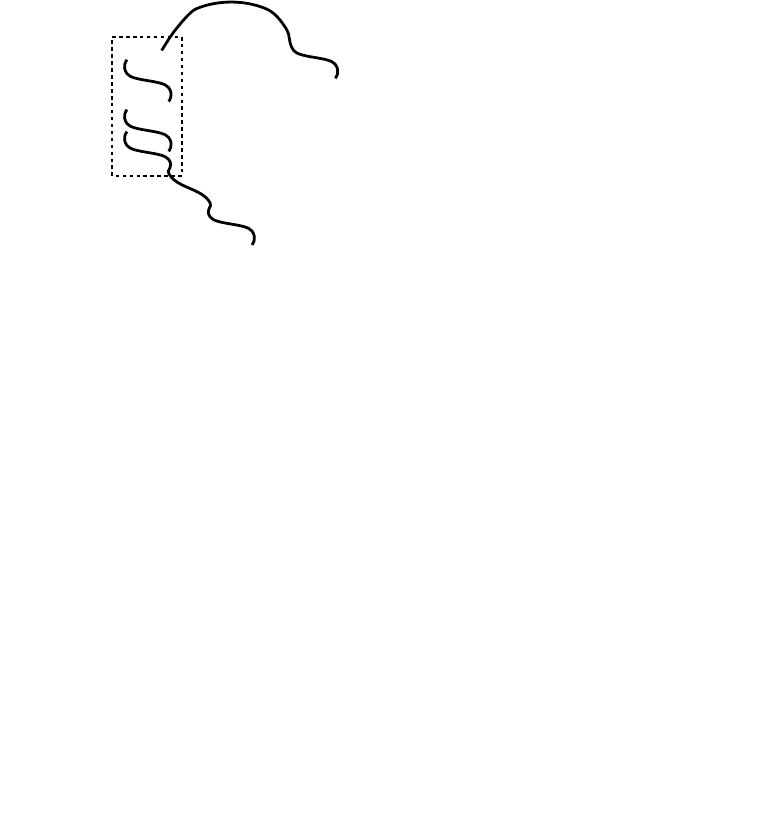
  \hspace{.2in}
  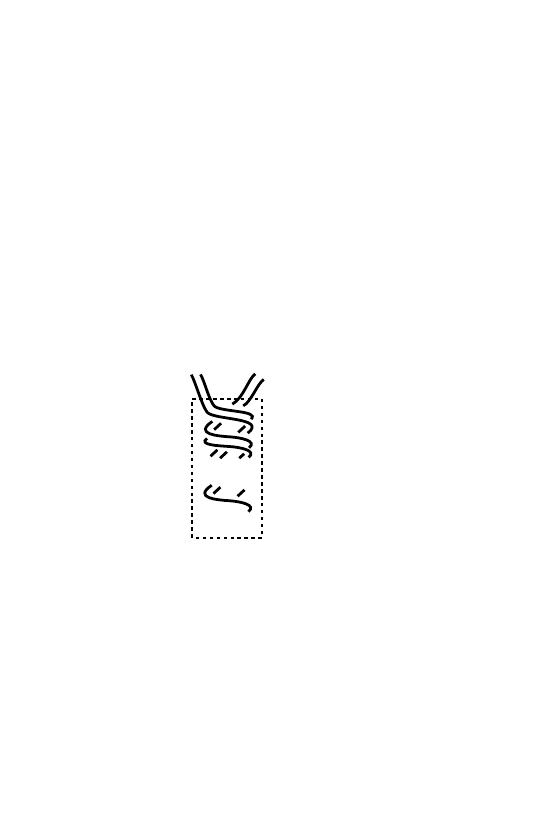
  \caption{Left: A highly twisted plat. Right: The link $\calL'$.}
  \label{fig:HighlyTwistedPlat}
\end{figure}

\begin{definition}\label{def:PlatProjection}
Fix $g \geq 1$. Let $K_g'$ be a knot given by a plat projection as above with $2m$ strands, where $m=g+2$. Let $r$, the number of rows of the projection, be fixed, equal to $4m(m-2)+1 = 4(g+2)g+1$. Finally, select the $a_{i,j}$ to satisfy the following.
\begin{itemize}
\item $|a_{i,j}| \geq 6$. 
\item If $i=1$, then $a_{1,j}$ is odd.
\item If $i>1$, then $a_{i,j}$ is even.
  \end{itemize}
\end{definition}

\begin{lemma}\label{lemma:PlatsAreNice}
The link $K_g'$ has the following properties for any $g\geq 1$. 
\begin{enumerate}
\item $K_g'$ is a knot (i.e.\ a single link component).
\item The bridge number of $K_g'$ is $g+2$, and the bridge sphere is unique up to isotopy (any bridge sphere is isotopic to the horizontal one in the plat projection). 
\item The tunnel number of $K_g'$ is $g+1$.
\item $K_g'$ is hyperbolic.
\end{enumerate}
\end{lemma}

\begin{proof}
The fact that $K_g'$ is a knot follows from the fact that there are an odd number of crossings on the top row, and an even number of crossings thereafter. This connects the strands into a single component.

The rest of the items follow from appeals to several references. First, by work of Johnson and Moriah~\cite{JohnsonMoriah}, the distance of the induced bridge sphere $\Sigma$ (see~\cite{JohnsonMoriah} for definitions) is $\lceil r/(2(m-2)) \rceil > 2m$.  Work of Tomova~\cite{Tomova} then implies that $\Sigma$ is the unique minimal bridge sphere up to isotopy.  For the third claim, we first note that by using a generic operation called meridional stabilization on $\Sigma$, we get a genus $m$ Heegaard surface $Q$ for $E(K_g')$.  In fact,~\cite{Tomova} also asserts that if $Q'$ is any Heegaard surface for $E(K_g')$ such that $\chi(Q') \leq 2+2m$, then $Q'$ is related to $\Sigma$ by a sequence of generic operations and thus $g(Q') \geq m$.  It follows that $g(E(K_g')) = m = g+2$, and thus the tunnel number of $K_g'$ is $g+1$.

Finally, the fact that $K_g'$ is hyperbolic follows from \cite[Theorem~1.2]{FuterPurcell}, using the hypothesis that each $|a_{i,j}|\geq 6$. The fact that the plat diagram is prime and twist reduced is straightforward, and we leave its proof to the reader. (Alternatively, $K_g'$ is obtained by Dehn filling a hyperbolic link $L_g$ described in Section~\ref{sec:proof}. In \cite{FuterPurcell}, it is shown that provided $|a_{i,j}|\geq 6$, each slope has length at least six, and so the 6-Theorem implies $K_g'$ is hyperbolic \cite{agol:bounds, lackenby:word}.)
\end{proof}

Let $K_g'$ be a knot as in Definition~\ref{def:PlatProjection}.  We will use $K_g'$ to produce the sequence of knots $\{K_n\}$ with unbounded genus $g$ bridge numbers.  To begin, let $L_1'$ and $L_2'$ be the pushoffs of $K_g'$ which lie flat in the projection plane following $K_g'$, except for two modifications:
\begin{enumerate}
\item In each twist region, the four parallel strands of $L_1\cup L_2$ contain the corresponding number $a_{i,j}$ of half twists.
\item At the bottom of the first column, the two parallel strands of $L_1'$ and $L_2'$ have 14 positive crossings.
\end{enumerate}
A depiction of $L_1'$ and $L_2'$ is shown in Figure~\ref{fig:HighlyTwistedPlat}, right.  Note that $L_1'$ and $L_2'$ cobound an annulus $R'$ containing $K_g'$ as a core.  Take $K'$ to be a simple unknot in the lower left-hand corner of the diagram that meets $R'$ twice and links with each of $L_1'$ and $L_2'$ once, with one intersection of $K'$ and $R'$ separating the 14 crossings into two pairs of 7, as shown in Figure~\ref{fig:HighlyTwistedPlat}, right.  Let $\mathcal L'$ be the 3-component link $K' \cup L_1' \cup L_2'$.

Let $\hat{R}'$ denote the twice-punctured annulus $R' \cap E(\mathcal L')$.  In~\cite{BakerGordonLuecke2013}, the authors define a \emph{catching surface} for the pair $(\hat{R}',K')$.  In the present work, we use a more specific definition from~\cite{BowmanTaylorZupan} which will suffice for our purposes.

\begin{definition}
An orientable, connected, properly embedded surface $Q \subset E(\mathcal L')$ is a \emph{catching surface} for $(\hat{R},K')$ if
\begin{enumerate}
\item $\chi(Q) < 0$
\item $Q \cap \partial N(L_i')$ is a nonempty collection of coherently oriented essential curves, and
\item curves of $Q \cap \partial N(L_i')$ meet curves of $\hat{R}' \cap \partial N(L_i')$.
\end{enumerate}
\end{definition}

\begin{lemma}\label{lemma:QCatchingSurface}
  The 2-punctured disk $Q'$ bounded by $K'$ in $E(\mathcal L')$ is a catching surface for $(\hat{R}',K')$. Moreover, $\partial Q'$ is meridional on $N(L_1')$ and $N(L_2')$, and meets a meridian of $N(K')$ exactly once.\qed
\end{lemma}

We use the following theorem from \cite{BakerGordonLuecke2013}, as stated in \cite{BowmanTaylorZupan}.

\begin{theorem}[\cite{BakerGordonLuecke2013}, Theorem~1.2]\label{thm:BakerGordonLuecke}
Let $\mathcal{L} = K\cup L_1 \cup L_2$ be a link in $S^3$, and let $R$ be an annulus in $M=S^3$ with $\bdy R = L_1 \cup L_2$. Assume $(R,K)$ is caught by a surface $Q$ in $E(\mathcal{L})$ with $\chi(Q)<0$. Suppose that $\bdy Q$ is meridional on $N(L_1)$ and $N(L_2)$ and meets a meridian of $N(K)$ exactly once. Let $K^n$ be $K$ twisted $n$ times along $R$. If $H_1\cup_\Sigma H_2$ is a genus $g$ Heegaard splitting of $S^3$, then either
\begin{enumerate}
\item $R$ can be isotoped to lie in $\Sigma$,
\item there is an essential annulus properly embedded in $E(\mathcal{L})$ with one boundary component in each of $N(L_1)$ and $N(L_2)$, or
\item for each $n$,
  \[
  b_g(K^n) \geq \frac{1}{2}\left( \frac{n}{-36\chi(Q)} - 2g + 1\right).
  \]
\end{enumerate}
\end{theorem}

It follows that

\begin{corollary}\label{cor:UnboundedBridge}
If $\mathcal{L}'$ is hyperbolic and $K^n$ is the result of performing $n$ twists on $K'$ along $R'$, then $b_g(K^n) \rightarrow \infty$ as $n \rightarrow \infty$.
\end{corollary}
\begin{proof}
Since the 2-punctured disk $Q'$ is a catching surface for $(\hat{R},K')$, one of the three conclusions of Theorem~\ref{thm:BakerGordonLuecke} holds.  By Lemma~\ref{lemma:PlatsAreNice}, the annulus $R'$ is not isotopic into the genus $g$ surface $\Sigma$; otherwise, the tunnel number of $K_g$ is at most $g$ (see, for example,~\cite{morimoto}). So conclusion (1) does not hold.  By assumption, $\mathcal L'$ is hyperbolic, and thus conclusion (2) does not hold.  Thus, for each $n$
\[ b_g(K^n) \geq \frac{1}{2}\left(\frac{n}{36}-2g+1\right).\qedhere \]
\end{proof}

\section{Proof of the main theorem}\label{sec:proof}

In order to demonstrate that the constructed links $\mathcal{L'}$ satisfy the hypotheses of Corollary~\ref{cor:UnboundedBridge}, we construct a hyperbolic surgery parent $\mathcal{L}$ of the link $\mathcal{L}'$ and show that $E(\mathcal{L'})$ may be obtained by hyperbolic Dehn filling along boundary components of $E(\mathcal{L})$.

To construct the link $\mathcal{L}$, we begin with the same knot template which produced knot $K_g'$ in Section~\ref{sec:LinkDesc}.  Let $K_g$ denote a knot with a twisted plat projection with the same row and column parameters $m$ and $r$ as $K_g'$ but with twisting parameters $a_{i,j}$ such that $a_{1,j} = 1$ and $a_{i,j} = 0$ if $i >1$.  For the $(i,j)$-th twist region of $K_g$, consider the simple closed curve $C_{i,j}$, unknotted in $S^3$, encircling exactly the two strands of the twist region. Such a curve is called a \emph{crossing circle}.  Let $L_g$ denote the union of $K_g$ and the crossing circles $C_{i,j}$, shown in Figure~\ref{fig:AugHighlyTwistedPlat}, left.  The construction of $L_g$ is standard, and $L_g$ is called a \emph{fully augmented link} as in \cite{Purcell:IntroAug}.  It is hyperbolic by work of Adams \cite{Adams:Aug}. Observe that $E(L_g)$ is homeomorphic to the complement of $E(L_g')$, the fully augmented link corresponding to the link $K_g'$.

\begin{figure}
  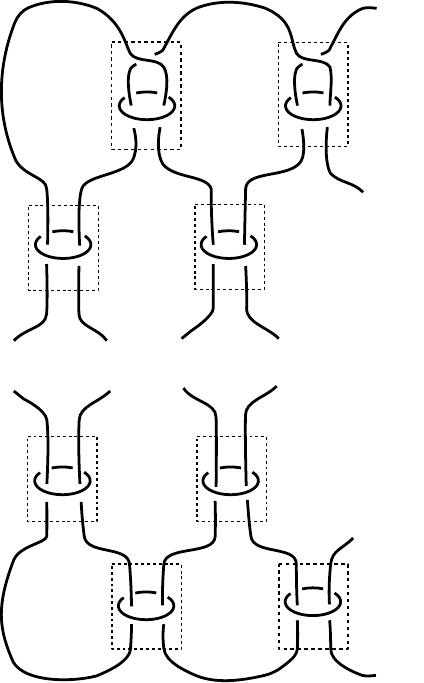
  \hspace{1in}
  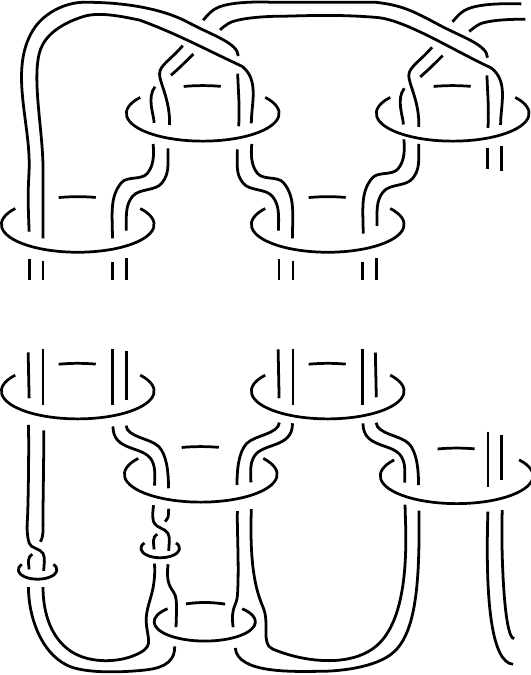
  \caption{Left: The fully augmented link $L_g$. Right: The generalized augmented link $\calL$.}
  \label{fig:AugHighlyTwistedPlat}
\end{figure}

Similar to the construction in Section~\ref{sec:LinkDesc}, we take $L_1$, $L_2$ to be parallel to $K_g$, bounding an annulus $R$ between them which contains $K_g$ as a core.  Let $L_1$ and $L_2$ be the pushoffs of $K_g$ which lie flat in the projection plane following $K_g$, except for two modifications:
\begin{enumerate}
\item In the first row of twist regions, components $L_1$ and $L_2$ form a half twist of four parallel strands, and
\item At the bottom of the first column, $L_1$ and $L_2$ have two positive crossings.
\end{enumerate}
For each of the two positive crossings, introduce crossing circles $C_1$ and $C_2$.  Finally, let $K$ be a simple unknot lying in the lower left corner of the diagram that pierces the annulus $R$ in two points, links each of $L_1$ and $L_2$ once, and such that one intersection of $K$ and $R$ occurs between the two positive crossings specified above.  We define the link $\mathcal{L}$ to be
\[ \mathcal{L} = L_1 \cup L_2 \cup K \cup C_1 \cup C_2  \cup \left(\bigcup C_{i,j} \right).\]
The link $\mathcal{L}$ is shown in Figure~\ref{fig:AugHighlyTwistedPlat}, right. It is a \emph{generalized augmented link}, as in \cite{Purcell:Slope, Purcell:MultiplyTwistedHyp, Purcell:MultiplyTwistedSF}.

\begin{lemma}\label{lemma:Fill}
The link exterior $E(\mathcal{L'})$ can be obtained by Dehn filling $E(\mathcal{L})$.
\end{lemma}
\begin{proof}
Corresponding to the crossing circles $C_1$ and $C_2$, perform $1/3$-sloped Dehn filling, which results in adding six new crossings to each existing crossing, so that each twist region contains seven crossings.  Next, corresponding to each crossing circle $C_{1,j}$ (i.e.\ in the first row) perform $1/((a_{1,j}-1)/2)$ Dehn filling.  This creates $a_{1,j}$ half twists in each generalized twist region of the first row.  Finally, corresponding to all other crossing circles $C_{i,j}$, perform $1/(a_{i,j}/2)$ Dehn filling, adding $a_{i,j}$ half twists to the relevant twist region.  By construction, the link component $K$ is unaffected by these Dehn fillings, and the resulting 3-manifold is the exterior of $E(\mathcal{L'})$. 
\end{proof}

\begin{proposition}\label{prop:LHyp}
  The generalized augmented link $\mathcal{L}$ is hyperbolic. 
\end{proposition}

The proof of Proposition~\ref{prop:LHyp} is long, so we postpone it until the next section. Meanwhile, assuming that result, we continue. 

\begin{corollary}\label{cor:LHyp}
  The link $\mathcal{L}'$ is hyperbolic.
\end{corollary}
\begin{proof}
  By Proposition~\ref{prop:LHyp}, the link $\mathcal{L}$ is hyperbolic.  Moreover, each of the Dehn fillings specified in Lemma~\ref{lemma:Fill} yields a generalized twist region containing at least six half twists. As in the proof of  \cite[Theorem~3.2]{Purcell:MultiplyTwistedHyp}, by \cite[Proposition~3.1]{Purcell:MultiplyTwistedHyp}, each slope of the Dehn filling has length at least six, and so the 6--Theorem implies the Dehn filling is hyperbolic \cite{agol:bounds, lackenby:word}. Thus $E(\calL')$ is hyperbolic. 
\end{proof}

We now conclude the proof of Theorem~\ref{thm:BoundedVolUnboundedBridge}.

\begin{proof}[Proof of Theorem~\ref{thm:BoundedVolUnboundedBridge}]
Let $K^n$ be the sequence of knots obtained by twisting $K'$ along $\hat{R}'$, noting that this twisting can be achieved by Dehn filling $E(\mathcal{L'})$ along $\partial N(L_1')$ and $\partial N(L_2')$.  Since only finitely many Dehn fillings yield non-hyperbolic 3-manifolds, there is a threshold $n^*$ such that $K^n$ is hyperbolic whenever $n > n^*$.  Note that since each $K^n$ is obtained by Dehn filling the same parent link, the volume of any $K^n$ is bounded by the volume of the parent link. In particular, $\vol(K^n) \leq \vol(\mathcal{L}') \leq \vol(\mathcal{L})$.

On the other hand, by Corollary~\ref{cor:UnboundedBridge}, we have $b_g(K^n) \rightarrow \infty$ as $n \rightarrow \infty$.
\end{proof}

\section{The generalized augmented link $\calL$ is hyperbolic}

In this section, we complete the proof of Proposition~\ref{prop:LHyp}, that $\calL$ is a hyperbolic generalized augmented link.

\begin{lemma}\label{lemma:Irreducible}
  The manifold $E(\calL)$ is irreducible and boundary irreducible.
\end{lemma}

\begin{proof}
Let $S$ be an embedded 2-sphere in $E(\calL)$. Recall that the link $L_g$ is hyperbolic, and since $L_1$ is isotopic to $K_g$ away from the crossing circles, the link $L_g$ is isotopic to the union of $L_1$ and the crossing circles.  In addition, $L_g$ may be viewed as sublink of $\calL$ (by replacing $K_g$ with $L_1$), and so $E(\calL)$ embeds in $E(L_g)$; hence $S$ embeds in $E(L_g)$.  As such, $S$ cannot separate any component of $L_g$.  Switching the roles of $L_1$ and $L_2$ in $L_g$, we see that $S$ cannot separate $L_2$ from the components of $L_g$ either.  Further, $C_1$, $C_2$, and $K$ have nonzero linking number with $L_1$, and so these are also on the same side of $S$. It follows that $S$ bounds a ball in $E(\calL)$, so that manifold is irreducible. 

Similarly, by hyperbolicity of the augmented link $L_g$, no component of $L_g \cup L_2$ bounds a disk, and none of $C_1$, $C_2$, and $K$ bound disks by a linking number argument.  This implies $E(\calL)$ is boundary irreducible. 
\end{proof}

\subsection{Ideal polyhedral decomposition}

Next, we show that $E(\calL)$ is atoroidal. This requires results from~\cite{Purcell:Slope,Purcell:MultiplyTwistedSF,Purcell:MultiplyTwistedHyp}. In particular, the link $\calL$ is a generalized fully augmented link, and so it admits some nice properties.

\begin{lemma}\label{lemma:calL'Reflect}
  The link complement $E(\calL)$ admits an orientation reversing involution fixing a surface $P$ pointwise. The annulus $R$ is a component of $P$. 
\end{lemma}

\begin{proof}
  This follows from Proposition~3.1 of \cite{Purcell:Slope}. Briefly, the involution is given by a reflection through the plane of projection, followed by a homeomorphism that puts a full twist at crossing circles bounded by a half twist, giving back the original diagram. Away from half twists, this fixes the plane of projection pointwise. In a neighborhood of half twists, this fixes a surface $P$ that is also half twisted. The annulus $R$ is fixed pointwise by the involution, hence forms a component of $P$. 
\end{proof}

Let $D_{i,j}$ in $E(\calL)$ denote the disk bounded by the crossing circle $C_{i,j}$. Let $D_1$ and $D_2$ denote disks bounded by $C_1$ and $C_2$, respectively, and let $D$ be the disk bounded by $K$. We let $\mathcal{D}$ be the collection of all these disks. Note we may take these disks to be disjoint, and to meet $L_1$ and $L_2$ in as few points as possible. In particular, $D_{i,j}$ meets each $L_k$ twice, and $D_i$ and $D$ meet each $L_k$ once. 

\begin{lemma}\label{lemma:Polyhedra}
The link complement $E(\calL)$ admits a decomposition into two identical ideal polyhedra. The ideal polyhedra are checkerboard colored, with shaded faces coming from disks in $\mathcal{D}$, and white faces coming from $P$. In each polyhedron, there is a white face for each complementary region of the diagram of $\calL'$, and two shaded faces for each crossing circle. 
\end{lemma}

\begin{proof}
This is as in \cite{Purcell:Slope, Purcell:MultiplyTwistedSF}. Shade
each disk in $\mathcal{D}$, and then cut along it, giving two boundary
components homeomorphic to the original disk. Where the disk meets a
half twist, the cut result is homeomorphic to removing the half twist
by untwisting one side only. After cutting and untwisting, the resulting manifold with boundary can be isotoped to be symmetric with respect to reflection in the plane of projection; in fact, untwisting takes $P$ to the plane of projection. Cut along the plane of projection. Obtain two polyhedra, with remnants of link components corresponding to ideal vertices.

Note faces come from disks bounded by crossing circles and the surface $P$, as claimed; each crossing circle gives two shaded faces since we sliced the disk apart. Each region of the diagram gives a white face. Also note that edges lie on the intersection of crossing disks in $\mathcal{D}$ with the projection plane. Each edge meets a white face and a shaded face. Thus the polyhedron is checkerboard colored. 
\end{proof}

Let $P_1$ and $P_2$ denote the identical polyhedra of Lemma~\ref{lemma:Polyhedra}. The form of one of the polyhedra is shown in Figure~\ref{fig:Polyhedron}.

\begin{figure}
  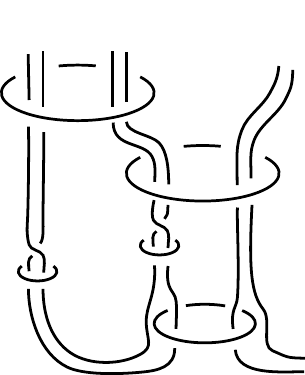
  \hspace{.5in}
  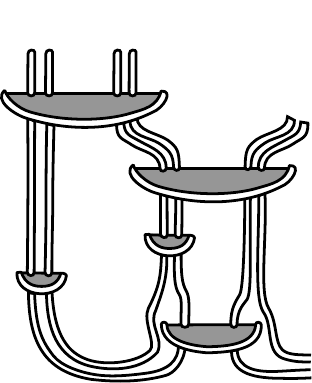
  \hspace{.5in}
  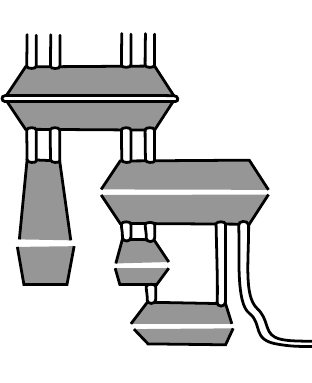
  \caption{The form of a polyhedron coming from $\calL'$. On the right, remnants of the link correspond to ideal vertices.}
  \label{fig:Polyhedron}
\end{figure}

\begin{lemma}\label{lemma:TorusForm}
Suppose $T$ is an essential torus in $E(\calL)$. Then $T$ can be isotoped to have the following properties.
\begin{enumerate}
\item $T$ is preserved under the involution of Lemma~\ref{lemma:calL'Reflect}.
\item $T$ meets each face of $P_1$ and $P_2$ in arcs with endpoints on distinct edges. That is, there are no simple closed curves of intersection of $T$ with faces, and there are no arcs of intersection with both endpoints on the same edge of the polyhedron.
\item The intersection of $T$ with surfaces $\mathcal{D}$ and $P$ forms a graph on $T$ in which all regions are quadrilaterals, with opposite edges on each quadrilateral coming from shaded faces or white faces.
\end{enumerate}
\end{lemma}

\begin{proof}
This can be found in \cite{Purcell:MultiplyTwistedSF}; we review the proof briefly. An essential torus $T$ must meet faces of $P_j$, else it is contained in a ball and not essential. If $T$ meets only faces of $\mathcal{D}$, then it meets them in simple closed curves bounding disks, which can be pushed off. Hence $T$ meets the surface $P$. Then the equivariant torus theorem~\cite{holzmann} implies that $T$ can be isotoped to be preserved under the involution. This gives item (1).

Item (2) follows by standard arguments involving isotopy of essential surfaces; we leave its proof to the reader.

Finally, item (3) follows by an Euler characteristic argument. Disks $\mathcal{D}$ intersect $P$ in edges; where $T$ meets such an edge we pick up a 4-valent vertex on $T$. Note that the regions on $T$ cannot be triangles, since at a vertex each region meets an edge coming from one shaded and one white face, nor bigons, since a white edge on $T$ coming from a bigon region would contradict item (2).
Let $e$ be the number of edges on $T$, $v$ the number of vertices, and $f$ the number of faces. Since each vertex is 4-valent, $e=2v$. Since each face has at least four edges, $2e\geq 4f$. Then
\[ 0 = v-e+f \geq e/2 - e + e/2 = 0. \]
Thus the inequality must be an equality, so 2e=4f and each face is a quad. 
\end{proof}

\subsection{Essential tori}
We now consider the form of possible essential tori in $E(\calL)$ in order to rule them out.  The strategy is to classify the components of a possible essential torus cut along the disks $\calD$ and then prove that any coherent gluing of these components yields a boundary parallel torus.

\begin{lemma}\label{lemma:CompOfT}
Suppose $T$ is an essential torus in $E(\calL)$. Then each component of $T$ cut along $\calD$ is one of the following:
\begin{enumerate}
\item (Type I) A tube around a strand of $(L_1 \cup L_2) - \calD$,
\item (Type II) A tube around two parallel strands of $(L_1 \cup L_2) - \calD$ connecting disks in $\calD - D$, or
\item (Type III) One of the two exceptional annuli $A_1$ or $A_2$ pictured in Figure~\ref{fig:Annuli}.
\end{enumerate}
\end{lemma}

\begin{figure}
  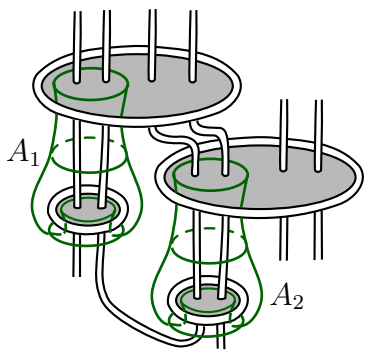
  \caption{The possible Type III components}
  \label{fig:Annuli}
\end{figure}
 
\begin{proof}
Let $P_1$ and $P_2$ denote the polyhedra given by Lemma~\ref{lemma:Polyhedra}.  Following Lemma~\ref{lemma:TorusForm}, the intersection of $T$ with the faces of $P_1$ and $P_2$ is a collection of quadrilaterals; hence we classify quads.  For this task, we construct a planar graph $\Gamma$ from $\partial P_1$ by crushing each shaded face in the checkerboard coloring (faces coming from $\calD$) to a vertex, and connecting two vertices with an edge when the two corresponding shaded faces have an ideal vertex in common. See Figure~\ref{fig:Graph}.  Note that this may give rise to parallel edges.

\begin{figure}
  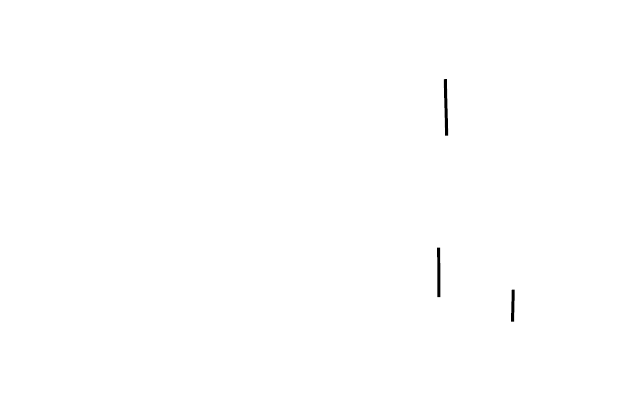
  \hspace{1in}
  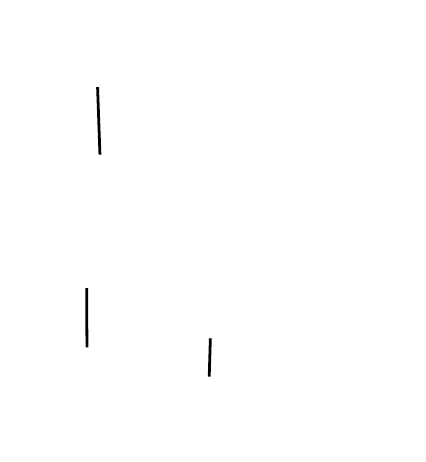
  \caption{Left: Crushing black faces to obtain the graph $\Gamma$. Right: The faces of $\Gamma$ sharing three vertices.}
  \label{fig:Graph}
\end{figure}

If $Q$ is a properly embedded quad component of $T \cap P_1$ in $P_1$, then this process crushes the two edges of $\partial Q$ in black faces, leaving the other two edges $q_1$ and $q_2$ connecting the same pair of vertices in $\Gamma$, and thus $q_1$ and $q_2$ are contained in faces $F_1$ and $F_2$ of $\Gamma$ which share at least two vertices.  On the other hand, observe that two such edges $q_1$ and $q_2$ in $\Gamma$ uniquely determine a quad in $P_1$.

Suppose the quad $Q$ is innermost in $P_1$, and suppose first that the curve $q = q_1 \cup q_2$ encloses a region containing no vertices and one or two edges in its interior. If $q_1$ and $q_2$ connect vertices corresponding to the same disk $D^* \in \calD$, then $q$ encloses a single edge corresponding to the ideal vertex of a crossing circle bounding $D^*$, and since $T$ is preserved by reflection, $T$ contains the union of two rectangles glued along $Q$.  Since $Q$ is innermost and has two boundary components in $D^*$, we see that $T$ is the union of the two rectangles and is parallel to $\partial D^*$, a contradiction. 
  
Now suppose that $q_1$ and $q_2$ connect vertices corresponding to distinct disks in $\calD$.  Then $q$ encloses either one or two edges, and since $T$ is preserved by reflection, the corresponding component of $T$ cut along $\calD$ is a tube around one or two parallel strands of $(L_1 \cup L_2) - \calD$. Note that there are no 2-strand tubes meeting the disk $D$ bounded by $K$; thus in the second case the 2-strand tubes connect distinct disks in $\calD - D$.

Finally, suppose that $q$ encloses a region which contains a vertex in its interior.
Then the faces $F_1$ and $F_2$ share two vertices which are separated by subgraphs containing at least one vertex each. By inspection, the only faces which meet at a pair of vertices of this form are the
two pairs of faces $(F,F')$, and $(F,F'')$ of $\Gamma$ that meet at more than two vertices; each pair has three vertices in common, and $q$ is as pictured in Figure~\ref{fig:Graph}, right.  Once again, $T$ is preserved by reflection, and so a component of $T$ cut along $\calD$ consists of two rectangles glued along $q$.  There are two possibilities for these components: the exceptional annuli $A_1$ and $A_2$ pictured in Figure~\ref{fig:Annuli}.
\end{proof}

\begin{proposition}\label{prop:L'Atoroidal}
  The link $\calL$ is atoroidal.
\end{proposition}

\begin{proof}
Let $T$ be an essential torus in $E(\calL)$.  By Lemma~\ref{lemma:CompOfT}, after isotopy $T - \calD$ consists of Type I, Type II, and Type III components.  Since a torus cannot be built from only Type III components, $T - \calD$ must contain a Type I or a Type II component.

Suppose first that $T - \calD$ contains a Type I component. Observe that
for each Type I component $\hat A$, and for each boundary component $\hat a$ of $\partial \hat A$, there is a unique component which extends $\hat A$ along $\hat a$, and that component is of Type I. It follows that the union of a maximal collection of Type I components of $T - \calD$ is equal to $T$, and $T$ is parallel to either $L_1$ or $L_2$, a contradiction.

On the other hand, suppose that $T - \calD$ contains a Type II component, and let $D_*$ and $D_{**}$ denote the two disks in $D_{i,j}$ which are adjacent to $D$; that is, there is a strand of $L_1$ or $L_2$ connecting these disks to $D$.  Now, for any Type II component $\hat A$ and component $\hat a$ of $\partial \hat A$ not contained in $D_*$, $D_{**}$, $D_1$, or $D_2$ (the disks adjacent to $D)$, there is a unique component $\hat A^*$, of Type II, extending $\hat A$ along $\hat a$, following the annulus $R$. It follows that the union of a maximal collection of Type II components connects $D_1$ to $D_*$ or $D_2$ to $D_{**}$.  Finally, observe that there is no component of any type extending a Type II component from $D_*$ or $D_{**}$ in the direction of $D$, contradicting Lemma~\ref{lemma:CompOfT}.  We conclude that $E(\calL)$ contains no essential tori.
\end{proof}

\subsection{Hyperbolicity}

We now finish the proof of Proposition~\ref{prop:LHyp}. 
The fact that the link contains no essential annuli follows from the previous results.

\begin{lemma}\label{lemma:AnAnnular}
  $E(\calL)$ is an-annular.
\end{lemma}

\begin{proof}
Assume $A$ is an essential annulus in $E(\calL)$. If $A$ has boundary components on two distinct link components, take the boundary of a small embedded neighborhood of $A$ and those link components. This is a torus $T$; by Proposition~\ref{prop:L'Atoroidal}, $T$ is inessential. It cannot be boundary parallel, because it contains two components on one side and the rest of the components on the other. Hence it is compressible. The compressing disk $\Delta$ cannot be on the side of the torus containing $A$, hence it lies on the other side. Surger along $\Delta$. The result is a sphere. By Lemma~\ref{lemma:Irreducible}, it bounds a ball in $E(\calL)$. Since the ball cannot be on the side of the sphere containing $A$, it must be on the other side. But then $E(\calL)$ has only two boundary components. This is a contradiction.

So assume $A$ has boundary components on the same link component $J$. The boundary of a small neighborhood of $A$ and $J$ gives two tori $T_1$ and $T_2$ in $E(\calL)$, neither of which is essential by Proposition~\ref{prop:L'Atoroidal}.  If $T_1$ is boundary parallel, then $T_1$ bounds a solid torus $V_1$ in $S^3$ containing either $J$ or another component $J'$ of $\calL$ at its core.  If $J'$ is a core of $V_1$, then $E(\calL)$ contains an essential annulus connecting $J$ and $J'$, reducing to the first case, a contradiction.  If $V_1$ contains $J$ at its core, then the annulus $A$ is boundary parallel, another contradiction.  Similarly, $T_2$ cannot be boundary parallel.

Thus both $T_1$ and $T_2$ are compressible, with compressing disks $\Delta_1$ and $\Delta_2$.
Note that $T_1$ is built of two annuli: an annulus on $\partial N(J)$ and the annulus $A$. Similarly for $T_2$. The disks $\Delta_1$ and $\Delta_2$ cannot have boundary meeting the annuli in closed curves, since $E(\calL)$ is boundary irreducible and $A$ is essential. Similarly, if $\partial \Delta_1$ or $\partial\Delta_2$ meets one of the annuli in an arc with both endpoints on the same boundary component of the annulus, then we may isotope the intersection away. Thus $\partial\Delta_1$ and $\partial\Delta_2$ meet the annuli 
in essential arcs. 
The disk $\Delta_1$ cannot intersect $J$. It follows that $\Delta_1$ lies on the side of $T_1$ that does not contain $J$, and similarly for $\Delta_2$. Now surger along $\Delta_1$, obtaining a sphere which must bound a ball by irreducibility. That ball cannot contain $J$. Hence $T_1$ bounds a solid torus $V_1$ embedded in $E(\calL)$. Similarly, $T_2$ bounds a solid torus $V_2$. Then $E(\calL)$ is obtained by gluing two solid tori $V_1$ and $V_2$ along a common annulus $A$ on their boundaries, and attaching the result to $\partial N(J)$. Therefore $E(\calL)$ has only one boundary component. This is a contradiction. 
\end{proof}

\begin{proof}[Proof of Proposition~\ref{prop:LHyp}]
By Lemma~\ref{lemma:Irreducible}, $E(\calL)$ is irreducible and boundary irreducible. By Proposition~\ref{prop:L'Atoroidal}, it is atoroidal, and by Lemma~\ref{lemma:AnAnnular} it admits no essential annuli. By work of Thurston~\cite{Thurston}, the link complement is hyperbolic. 
\end{proof}

\section{Fixed $b_g$ and unbounded volume}

\begin{proof}[Proof of Proposition~\ref{thm:UnboundedVolBoundedBridge}]
It is well-known that the collection of hyperbolic 2-bridge knots has unbounded volume, and $b_1(K) =1$ for a hyperbolic 2-bridge knot.  Thus, fix $g$ and $b$ such that $g+b \geq 3$, and for any $n$, let $K_n$ be a knot given by a highly twisted plat projection of $2m$ strands, where $m = g+b$, with $r_n$ rows, where $r_n > 4m(m-2)$ and $r_n \rightarrow \infty$ as $n \rightarrow \infty$.  Select twisting parameters $a_{i,j}$ so that $|a_{i,j}| \geq 7$ and so that $K_n$ is a knot, as in Lemma~\ref{lemma:PlatsAreNice}.

Then each $K_n$ is hyperbolic and the induced bridge sphere is distance at least $2m$.  Using the main result of~\cite{Tomova}, we have $b_g(K_n)=b$.  On the other hand, because there are at least seven crossings in each twist region, \cite[Theorem~1.2]{FKP:DehnFillingVolJP} implies that $\vol(K_n)$ is linearly bounded below by $\tw(K_n)$, the number of twist regions in the highly twisted plat projection of $K_n$.  Clearly, $r_n \rightarrow \infty$ implies that $\tw(K_n) \rightarrow \infty$, and thus $\vol(K_n) \rightarrow \infty$ as $n \rightarrow \infty$ as well.
\end{proof}

\bibliographystyle{amsplain}
\bibliography{references}

\end{document}